\documentclass[12pt]{article}
\def\date{31 July 2014}
\usepackage{amsmath, amssymb, amsfonts, theorem, color}

\newtheorem{proposition}{proposition}[section]

\newtheorem{theorem}[proposition]{Theorem}
\newtheorem{claim}[proposition]{Claim}

\newtheorem{lem}[proposition]{Lemma}

\newtheorem{thm}[proposition]{Theorem}

\newcommand{\qed}{$\square$\bigskip}

{\theorembodyfont{\rmfamily}
\newtheorem{definition}[proposition]{Definition}

}

\begin{document}
\font\smallrm=cmr8

\phantom{a}\vskip .25in
\centerline{{\large \bf  FIVE-LIST-COLORING GRAPHS ON SURFACES I.}}
\smallskip
\centerline{{\large\bf  TWO LISTS OF SIZE TWO IN PLANAR GRAPHS}}
\vskip.4in
\centerline{{\bf Luke Postle}%
\footnote{\texttt{luke@mathcs.emory.edu}.}} 
\smallskip
\centerline{Department of Mathematics and Computer Science}
\centerline{Emory University}
\centerline{Atlanta, Georgia  30323, USA}
\medskip
\centerline{and}

\medskip
\centerline{{\bf Robin Thomas}%
\footnote{\texttt{thomas@math.gatech.edu}. Partially supported by NSF under
Grant No.~DMS-1202640.}}
\smallskip
\centerline{School of Mathematics}
\centerline{Georgia Institute of Technology}
\centerline{Atlanta, Georgia  30332-0160, USA}

\vskip 1in \centerline{\bf ABSTRACT}
\bigskip

{
\parshape=1.0truein 5.5truein
\noindent
Let $G$ be a plane graph with outer cycle $C$, let $v_1,v_2\in V(C)$ and let $(L(v):v\in V(G))$ be 
a family of sets such that $|L(v_1)|=|L(v_2)|=2$, $|L(v)|\ge3$ for every $v\in V(C)\setminus\{v_1,v_2\}$
and  $|L(v)|\ge5$ for every $v\in V(G)\setminus V(C)$.
We prove a conjecture of Hutchinson that $G$ has a (proper) coloring $\phi$ such that
$\phi(v)\in L(v)$ for every $v\in V(G)$.
We will use this as a lemma in subsequent papers.
}

\vfill \baselineskip 11pt \noindent September 2012. Revised \date.
\vfil\eject
\baselineskip 18pt

\section{Introduction}

All graphs in this paper are finite and simple.
A \emph{list-assignment} for a graph $G$ is a family of non-empty sets $L=(L(v):v\in V(G))$.
An \emph{$L$-coloring}
of $G$ is a (proper) coloring $\phi$ such that $\phi(v)\in L(v)$ for all $v\in V(G)$. 
A graph is {\em $L$-colorable} if it has at least one $L$-coloring.
A graph $G$ is \emph{$k$-choosable}, also called \emph{$k$-list-colorable}, 
if $G$ has an $L$-coloring for every list-assignment $L$ for $G$ such that $|L(v)|\ge k$
for every $v\in V(G)$. 
List coloring was introduced and first studied by Vizing~\cite{Viz} and Erd\H{o}s, Rubin and Taylor~\cite{Erdos}.

Clearly every $k$-choosable graph is $k$-colorable, but the converse is false.
One notable example of this is that the Four-Color Theorem does not generalize to list-coloring. Indeed Voigt~\cite{Voigt} constructed a planar graph that is not $4$-choosable. 
On the other hand Thomassen~\cite{ThomPlanar} proved the following remarkable theorem with an outstandingly short proof.

\begin{theorem}[Thomassen]\label{PlanarChoosable}
Every planar graph is $5$-choosable.
\end{theorem}

Actually, Thomassen~\cite{ThomPlanar} proved a stronger theorem.

\begin{theorem}[Thomassen]\label{Thom}

If $G$ is a plane graph with outer cycle $C$, $P=p_1p_2$ is a subpath of $C$ of length one, and $L$ 
is a list assignment for $G$ such that $|L(v)|\ge 5$ for all $v\in V(G)\setminus V(C)$,
 $|L(v)|\ge 3$ for all 
$v\in V(C)\setminus V(P)$, $|L(p_1)|=|L(p_2)|=1$ and $L(p_1)\ne L(p_2)$, then $G$ is $L$-colorable.
\end{theorem}

Hutchinson~\cite{HutchOuterplanar} conjectured the following variation of Theorem~\ref{Thom}, which is the main result of this paper: 

\begin{theorem}\label{TwoTwos2}
If $G$ is a plane graph with outer cycle $C$, $v_1,v_2\in V(C)$ and $L$ is a list assignment for $G$
with $|L(v)|\ge 5$ for all $v\in V(G)\setminus V(C)$, $|L(v)|\ge 3$ for all $v\in V(C)\setminus \{v_1,v_2\}$, 
and $|L(v_1)|=|L(v_2)|=2$, then $G$ is $L$-colorable.
\end{theorem}

Hutchinson~\cite{HutchOuterplanar} proved Theorem~\ref{TwoTwos2} for outerplanar graphs. 
In fact, Theorem~\ref{TwoTwos2}  implies
Theorem~\ref{Thom}, as we now show. 
\bigskip

\noindent
{\bf Proof of Theorem~\ref{Thom}, assuming Theorem~\ref{TwoTwos2}.}
Let us assume for  a contradiction that
 $G,L$ and $P=p_1p_2$ give a counterexample to Theorem~\ref{Thom}, and the triple is chosen 
so that $|V(G)|$ is  minimum  and subject to 
that $|E(G)|$ maximum. 
It follows from the minimality of $G$ that the outer cycle $C$ of $G$ has no chords and that $G$ is $2$-connected. 
Since $C$ has no chords it follows that
for $i=1,2$ the vertex $p_i$ has a unique  neighbor in $C\setminus V(P)$, say $v_i$.
Let $G':=G\setminus V(P)$.
The graph $G'$ is $2$-connected, for otherwise it has a cutvertex, say $v$; but then $v$ is adjacent to $v_1$
by the maximality of $|E(G)|$, and hence $vv_1$ is a chord of $C$, a contradiction.
We deduce that $v_1\ne v_2$, for otherwise we could color $v_1$ using a color $c\not\in L(p_1)\cup L(p_2)$,
delete $v_1$, remove $c$ from the list of neighbors of $v_1$, and extend the coloring of $v_1$ to
an $L$-coloring of $G$ by the minimality of $G$, a contradiction.
Thus $v_1\ne v_2$.
Let $C'$ be the outer cycle of $G'$, and
for $v\in V(G)\setminus V(P)$ let $L'(v)$ be obtained from $L(v)$
by deleting $L(v_i)$ for all $i\in\{1,2\}$ such that $v_i$ is a neighbor of $v$.
Then $|L'(v)|\ge 3$ for all $v\in V(C')\setminus\{v_1,v_2\}$
and $|L'(v_1)|,|L'(v_2)|\ge2$.
By Theorem~\ref{TwoTwos2} applied to
$G', C', v_1,v_2$ and $L'$ the graph $G'$ has an $L'$-coloring.
It follows that $G$ has an $L$-coloring, a contradiction.~\qed

We will use Theorem~\ref{TwoTwos2} in subsequent papers to deduce various extensions of Theorem~\ref{Thom}
for paths $P$ of length greater than two.
In particular, we will prove the following theorem.

\begin{theorem}
\label{thm:pathbound}
If $G$ is a plane graph with outer cycle $C$, $P$ is a subpath of $C$ and $L$ is a list assignment
for $G$ with $|L(v)|\ge 5$ for all $v\in V(G)\setminus V(C)$, and $|L(v)|\ge 3$ for all $v\in V(C)\setminus V(P)$, 
then there exists a subgraph $H$ of $G$  with $|V(H)|=O(|V(P)|)$ such that for every $L$-coloring $\phi$ of $P$, 
either $\phi$ extends to an $L$-coloring of $G$ or $\phi$ does not extend to an $L$-coloring of $H$.
\end{theorem} 

\noindent
We need Theorem~\ref{thm:pathbound} and several similar results to show that graphs on a fixed surface
that are minimally not $5$-list-colorable satisfy  certain isoperimetric inequalities.
Those inequalities, in turn, imply several new and old results about $5$-list-coloring graphs on 
surfaces~\cite{PosPhD,PosThoHyperb}.

\section{Preliminaries}


\begin{definition}[Canvas]

We say that $(G, S, L)$ is a \emph{canvas} if $G$ is a plane graph, $S$ is a subgraph of the boundary of the 
outer face of $G$, and $L$ is a list assignment for the vertices of $G$ such that $|L(v)|\ge 5$ for all $v\in V(G)$ 
not incident with the outer face,
$|L(v)|\ge 3$ for all $v\in V(G)\setminus V(S)$, and there exists a proper $L$-coloring of $S$. 
%
\end{definition}



Thus Theorem~\ref{Thom} can be restated in the following slightly stronger form,
which follows easily from Theorem~\ref{Thom}.

\begin{theorem}
\label{Thom2}
If $(G,P,L)$ is a canvas, where $P$ is a path of length one, then  $G$ is $L$-colorable.
\end{theorem}
 
%
%
%
It should be noted that Theorem~\ref{TwoTwos2} is not true when one allows three vertices with list of size two. 
Indeed, Thomassen~\cite[Theorem 3]{ThomWheels} characterized the canvases $(G,S,L)$ such that $S$ is a path
of length two and some $L$-coloring of $S$ does not extended to an $L$-coloring of $G$.
One of Thomassen's obstructions does not extend even when the three vertices in $S$ are given lists of size two. 
To prove Theorem~\ref{TwoTwos2} we will need the following lemma, a consequence 
of~\cite[Lemma~1]{ThomWheels} and ~\cite[Theorem~3]{ThomWheels}.

\begin{lem} \label{WheelUniqueColor}
Let $(G,P,L)$ be a canvas, where $G$ has outer cycle $C$, $P=p_1p_2p_3$ is a path on three vertices and
$G$ has no path $Q$ with ends $p_1$ and $p_3$ such that every vertex of $Q$ belongs to $C$ and is adjacent to $p_2$.
Then there exists at most one $L$-coloring of $P$ that does not extend to an $L$-coloring of $G$.
\end{lem}


\begin{definition}
Let $(G,S,L)$ be a canvas and let $C$ be the outer walk of $G$. 
We say a cutvertex $v$ of $G$ is \emph{essential} if whenever $G$ can be written as $G=G_1\cup G_2$,
where $V(G_1),V(G_2)\ne V(G)$ and $V(G_1)\cap V(G_2)=\{v\}$,  then $V(S)\not\subseteq V(G_i)$ for $i=1,2$.
Similarly, we say a chord $uv$ of $C$ is \emph{essential} if 
whenever $G$ can be written as $G=G_1\cup G_2$,
where $V(G_1),V(G_2)\ne V(G)$ and $V(G_1)\cap V(G_2)=\{u,v\}$,  then $V(S)\not\subseteq V(G_i)$ for $i=1,2$.
\end{definition}

\begin{definition}
We say that a canvas $(G,S,L)$ is \emph{critical} if there does not an exist an $L$-coloring of $G$ but for every edge 
$e\in E(G)\setminus E(S)$ 
there exists an $L$-coloring of $G\setminus e$.
\end{definition}

\begin{lem}\label{CanvasCritical}
If $(G,S,L)$ is a critical canvas, then
\begin{enumerate}
\item[\rm(1)] every cutvertex of $G$ and every chord of the outer walk of $G$ is essential, and
\item[\rm(2)] every cycle of $G$ of length at most four bounds an open disk containing no vertex of $G$.
\end{enumerate}
\end{lem}
\begin{proof}
To prove (1) suppose for a contradiction that the graphs $G_1,G_2$ satisfy the requirements in the 
definition of essential cutvertex or essential chord, except that $V(S)\subseteq V(G_1)$.
Since $(G,S,L)$ is critical, there exists an $L$-coloring $\phi$ of $G_1$. 
By Theorem~\ref{Thom2}, $\phi$ can be extended to $G_2$. Thus $G$ has an $L$-coloring, a contradiction. This proves (1).

Statement (2) is a special case of~\cite[Theorem~6]{Bohme}.
It can also be deduced from Theorem~\ref{Thom2}.~\qed
\end{proof}

\section{Proof of the Two with Lists of Size Two Theorem}

In this section, we prove Theorem~\ref{TwoTwos2} in the following stronger form.
We say that an edge $uv$ {\em separates} vertices $x$ and $y$ if $x$ and $y$ belong to different components of 
$G\setminus\{u,v\}$.

\begin{thm} \label{DemTwo}
Let $(G,S,L)$ be a canvas, where $S$ has two components: a path $P$ and an isolated vertex $u$ with $|L(u)|\ge 2$. 
Assume that if $|V(P)|\ge 2$, then $G$ is $2$-connected,  $u$ is not adjacent to an internal vertex of $P$ and
there does not exist a chord of the outer walk of $G$ with an end in $P$ which separates a  vertex of $P$ from $u$.
Let $L_0$ be a set of size two.
If $L(v)=L_0$ for all $v \in V(P)$, then $G$ has an $L$-coloring, unless $L(u)=L_0$ and 
$V(S)$ induces an odd cycle in $G$.
\end{thm}

\begin{proof}
Let us assume for a contradiction that $(G,S,L)$ is a counterexample with $|V(G)|$  minimum and subject to that
with $|V(P)|$  maximum. 
Hence $G$ is connected and $(G,S,L)$ is critical. 
Let $C$ be the outer walk of $G$. 
By the first statement of Lemma~\ref{CanvasCritical} all cutvertices of $G$ and all chords of $C$ are essential. 
Thus we have proved:

\begin{claim}
\label{cl:nochord}
There is no chord with an end in $P$.
\end{claim}

\begin{claim}
\label{cl:2conn}
$G$ is $2$-connected.
\end{claim}
\begin{proof}
Suppose there is a cutvertex $v$ of $G$. By assumption then, $|V(P)|=1$. 
Since $v$ is a cutvertex the graph $G$ can be expressed as $G=G_1\cup G_2$, where $V(G_1)\cap V(G_2)=\{v\}$ and
$V(G_1)\setminus V(G_2)$ and $V(G_2)\setminus V(G_1)$ are both non-empty.
As $v$ is an essential cutvertex of $G$, we may suppose without loss of generality that $u\in V(G_2)\setminus V(G_1)$ 
and $V(P)\subseteq V(G_1)\setminus V(G_2)$.

Consider the canvas $(G_1,S_1,L)$, where $S_1=P+v$, the graph obtained from $P$ by
adding $v$ as an isolated vertex.. 
As $|V(G_1)| < |V(G)|$, there exists an $L$-coloring $\phi_1$ of $G_1$. 
Let $L_1=(L_1(x):x\in V(G))$, where $L_1(v)=L(v)\setminus \{\phi_1(v)\}$ and $L_1(x)=L(x)$ for all $x\in V(G_1)\setminus \{v\}$.
Similarly, there exists an $L_1$-coloring $\phi_2$ of $G_1$ by the minimality of $G$. Note that $\phi_1(v)\ne \phi_2(v)$. 
Let $L_2=(L_2(x):x\in V(G_2))$, where 
$L_2(v)=\{\phi_1(v),\phi_2(v)\}$ and $L_2(x)=L(x)$ for all $x\in V(G_2)\setminus \{v\}$, and consider the canvas $(G_2,S_2,L_2)$,
where $S_2$ consists of the isolated vertices $v$ and $u$.
As $|V(G_2)| < |V(G)|$, there exists an $L_2$-coloring $\phi$ of $G_2$. Let $i$ be such that $\phi_i(v)=\phi(v)$. Therefore, $\phi\cup \phi_i$ is an $L$-coloring of $G$, contrary to the fact that $(G,S,L)$ is a counterexample. \qed
\end{proof}

Let $v_1$ and $v_2$ be the two (not necessarily distinct) vertices of $C$ adjacent to a vertex of $P$.
There are at most two such vertices by Claim~\ref{cl:nochord}.

\begin{claim}\label{NotBothU}
$v_1\ne v_2$.
\end{claim}
\begin{proof}
Suppose not; then $v_1=v_2=u$ and $V(S)=V(P)\cup\{u\}$ by Claim~\ref{cl:2conn}.
By Claim~\ref{cl:nochord} the graph $G[V(S)]$ is a cycle, and if it is odd, then $L(u)\setminus L_0\ne\emptyset$
by hypothesis. In either case the graph $G[V(S)]$ has an $L$-coloring $\phi$.
Let $G':=G\setminus V(P)$ and let $L'=(L'(x):x\in V(G'))$ be defined by $L'(x):=L(x)\setminus L_0$ if 
$x$ has a neighbor in $P$ and $L'(x):=L(x)$ otherwise.
By Theorem~\ref{Thom2} the graph $G'$ has an $L'$-coloring $\phi'$ with $\phi'(u)=\phi(u)$, and thus 
$G$ has an $L$-coloring, a contradiction.~\qed
\end{proof}


\begin{claim}\label{BadChord}
For all $i\in\{1,2\}$, if $v_i\ne u$, then $v_i$ is the end of an essential chord of $C$.
\end{claim} 
\begin{proof}
As $v_1$ and $v_2$ are symmetric, it suffices to prove the claim for $v_1$. So suppose $v_1\ne u$ and $v_1$ is not an end of an essential chord of $C$. First suppose that $|L(v_1)\setminus L_0| \ge 2$. Let $G'=G\setminus V(P)$ and let $S'$ consist of the isolated vertices $v_1$ and $u$. Furthermore, let $L'(v_1)$ be a subset of size two of $L(v_1)\setminus L_0$, 
let $L'(x) := L(x)\setminus L_0$ for all
vertices $x\in V(G')\setminus \{v_1,v_2\}$ with a neighbor in $P$ and let $L'(x):=L(x)$ otherwise. 
Note that the canvas $(G', L', S')$ satisfies the hypotheses of Theorem~\ref{DemTwo}. Hence as $|V(G')| < |V(G)|$ and $(G,S,L)$ is a minimum counterexample, it follows that $G'$ has an $L'$-coloring $\phi'$. Since $\phi'$ can be extended to $P$, $G$ has an $L$-coloring, a contradiction.

So we may assume that $L_0\subseteq L(v_1)$ and $|L(v_1)|=3$. Let $P'$ be the path obtained from $P$ by adding $v_1$.
 Let $S'=P'+u$, and let $L'=(L'(x);x\in V(G))$, where
$L'(v_1)=L_0$ and $L'(x)=L(x)$ for all $x\in V(G)\setminus \{v_1\}$. Consider the canvas $(G,S',L')$. As $v_1$ is not the end of an essential chord of $C$ and $(G,S,L)$ was chosen so that $|V(P)|$ was maximized, we find that $G[V(S')]$ is an odd cycle and $L(u)=L_0$. 

Now color $G$ as follows. Let $\phi(v_1)\in L(v_1)\setminus L_0$;
then we can extend $\phi$ to a coloring of $G[V(S')]$. 
Let $L'(v_1)=\{\phi(v_1)\}$, and for $x\in V(G)\setminus V(S')$ let
$L'(x)=L(x)\setminus L_0$ if $x$ has a neighbor in $S$ and let $L'(x)=L(x)$ otherwise. 
By Theorem~\ref{Thom2}, there exists an $L'$-coloring of $G\setminus V(S)$ and hence $\phi$ can be extended to
an $L$-coloring of $G$, a contradiction.~\qed
\end{proof}

By Claim~\ref{NotBothU} we may assume without loss of generality that $v_1\ne u$. By Claim~\ref{BadChord}, $v_1$ is an end of an essential chord of $C$. But this and Claim~\ref{cl:nochord} imply that $v_2\ne u$. By Claim~\ref{BadChord}, $v_2$ is an end of an essential chord of $C$. As $G$ is planar, it follows from Claim~\ref{cl:2conn} that $v_1v_2$ is a chord of $C$. 

\begin{claim}
$|V(P)|=1$
\end{claim}
\begin{proof}
Suppose not. 
Let $G_1,G_2$ be subgraphs of $G$ such that $G=G_1\cup G_2$, $V(G_1)\cap V(G_2)=\{v_1,v_2\}$,
$V(P)\subseteq V(G_1)$ and $u\in V(G_2)$. 
Let $v\not\in V(G)$ be a new vertex and
construct a new graph $G'$ with $V(G')=V(G_2)\cup \{v\}$ and $E(G')=E(G_2)\cup \{vv_1,vv_2\}$.  Let $L(v)=L_0$. 
Consider the canvas $(G',S',L)$, where $S'$ consists of the isolated vertices $v$ and $u$. As $|V(P)|\ge 2$, $|V(G')| < |V(G)|$. 
By the minimality of $(G,S,L)$, there exists an $L$-coloring $\phi$ of $G'$. 
Hence there exists an $L$-coloring $\phi$ of $G_2$, where $\{\phi(v_1),\phi(v_2)\}\ne L_0$. 
We extend $\phi$ to an $L$-coloring of $P\cup G_2$. 
Let  $L'(v_1)=\{\phi(v_1)\}$ and $L'(v_2)=\{\phi(v_2)\}$, and for $x\in V(G_1)\setminus (V(P)\cup\{v_1,v_2\})$
let $L'(x)=L(x)\setminus L_0$ if $x$ has a neighbor in $P$ and let $L'(x)=L(x)$ otherwise. 
By Theorem~\ref{Thom2}, there exists an $L'$-coloring $\phi'$ of $G_1\setminus V(P)$. As $\phi'$ can be extended to $P$, $G$ has an $L$-coloring, a contradiction. \qed
\end{proof}

Let $v$ be such that that $V(P)=\{v\}$. 

\begin{claim}
For $i\in\{1,2\}$, $L_0\subseteq L(v_i)$ and $|L(v_i)|=3$.
\end{claim}
\begin{proof}
By symmetry it suffices to prove the claim for $v_1$. 
If $|L(v_1)|\ge 4$, then let $c\in L_0$. 
If $|L(v_1)|=3$, then we may assume for a contradiction that $L_0\setminus L(v_1)\ne\emptyset$.
In that case  let $c\in L_0\setminus L(v_1)$. 

In either case, let $L'(v_1)=L(v_1)\setminus \{c\}$, $L'(v_2)=L(v_2)\setminus \{c\}$ and $L'(x)=L(x)$ otherwise. 
Consider the canvas $(G',S',L')$, where $G'=G\setminus\{v\}$ and $S'$ consists of the isolated vertices $v_2$ and $u$. 
As $|V(G')|<|V(G)|$, there exists an $L'$-coloring $\phi'$ of $G'$ by the minimality of $(G,S,L)$. 
Now $\phi'$ can be extended to an $L$-coloring of $G$ by letting $\phi'(v)=c$, a contradiction.~\qed
\end{proof}

\begin{claim}
$L(v_1)=L(v_2)$
\end{claim}
\begin{proof}
Suppose not. As $G$ is planar, either $v_1$ is not an end of a chord of $C$ separating $v_2$ from $u$, or $v_2$ is an the end of a chord separating $v_1$ from $u$. Assume without loss of generality that $v_1$ is not in a chord of $C$ separating $v_2$ from $u$. This implies that $v_1$ is not an end of a chord in $C$ other than $v_1v_2$. Let $v'$ be the vertex in $C$ distinct from $v_2$ and $v$ that is adjacent to $v_1$.

Let $c\in L(v_1)\setminus L_0$. Let $G'=G\setminus \{v,v_1\}$, $L'(x)=L(x)\setminus \{c\}$ if $x$ is adjacent to $v_1$ and $L'(x)=L(x)$ otherwise. 
Note that $|L'(v_2)|\ge 3$ as $L(v_1)\ne L(v_2)$ and $L_0\subseteq L(v_1)\cap L(v_2)$. 
Let $S'$ consist of isolated vertices $v'$ and $u$. 
By considering the canvas $(G',S',L')$ we deduce that $G'$ has an $L'$-coloring $\phi'$;
if $u\ne v'$, then it follows by the minimality of $G$, because in that case $|L(u)|,|L(v')|\ge 2$;
and if $u=v'$, then it follows from Theorem~\ref{Thom2}.
As $\phi'$ can be extended to $\{v,v_1\}$, there exists an $L$-coloring of $G$, a contradiction. \qed
\end{proof}



\begin{claim}
One of $v_1,v_2$ is the end of an essential chord of $C$ distinct from $v_1v_2$.
\end{claim}
\begin{proof}
Suppose for a contradiction that there is no such essential chord. 
Let $c\in L(v_1)\setminus L_0 = L(v_2)\setminus L_0$, and let $L_1$ be a set of size two such that
$c\in L_1\subseteq L(v_1)$ and $L_0\ne L_1$.
Let $L_1(v_1)=L_1(v_2)=L_1$ and $L_1(x)=L(x)$ for all $x\in V(G)\setminus \{v,v_1,v_2\}$.
Let $P'$ denote the path with vertex-set $\{v_1,v_2\}$ and
consider the canvas $(G\setminus v, P'+u, L_1)$. Note that $G\setminus v$ is $2$-connected, 
since $G$ is $2$-connected and there are no vertices in the open disk bounded by the triangle  $vv_1v_2$  by the second assertion of Lemma~\ref{CanvasCritical}. Since $P'$ has no internal vertex, the canvas $(G\setminus v, P'+u, L_1)$
 satisfies the hypotheses of Theorem~\ref{DemTwo}. 
As $|V(G')| < |V(G)|$, 
there exists an $L_1$-coloring $\phi'$ of $G\setminus v$. But then $\phi'$ can be extended to an $L$-coloring of $G$, a contradiction. \qed
\end{proof}

Suppose without loss of generality that $v_2$ is the end of an essential chord of $C$ distinct from $v_1v_2$. Choose such a chord $v_2u_1$ such that $u_1$ is closest to $v_1$ measured by the distance in $C\setminus v_2$. Let $G_1$ and $G_2$ be connected subgraphs of $G$ such that $V(G_1)\cap V(G_2)=\{v_2,u_1\}$, $G_1\cup G_2=G$,  $v\in V(G_1)$ and $u\in V(G_2)$. 

We now select an element $c$ as follows.
If $v_1$ is adjacent to $u_1$, then let $c\in L(v_1)\setminus L_0=L(v_2)\setminus L_0$.
Note that in this case $V(G_1)=\{v,v_1,v_2,u_1\}$ by the second assertion of Lemma~\ref{CanvasCritical}.
If $v_1$ is not adjacent to $u_1$, then we consider the canvas $(G_1,P'',L)$, where $P''=vv_2u_1$. 
As $u_1$ is not adjacent to $v_1$, there does not exist a path $Q$ in $G_1$ as in Lemma~\ref{WheelUniqueColor}. By Lemma~\ref{WheelUniqueColor}, there is at most one coloring of $P''$ which does not extend to $G_1$.
If such a coloring exists, then let $c$ be the color of $u_1$ in that coloring; otherwise let $c$ be arbitrary.

Consider the canvas $(G_2,S',L')$, where $S'$ consists of the isolated vertices $u_1$ and $u$, 
$L'(u_1)=L(u_1)\setminus \{c\}$ and $L'(x)=L(x)$ otherwise. 
As $|V(G_2)|<|V(G)|$, there exists an $L'$-coloring $\phi$ of $G_2$ by the minimality of $(G,S,L)$. But then we may extend $\phi$ to $G_1$
by the choice of $c$ to obtain an $L$-coloring of $G$, a contradiction.~\qed 
%
\end{proof}

%

\begin{thebibliography}{99}

\def\JCTB{{\it J.~Combin.\ Theory Ser.\ B}}
\def\CMUC{{\it Comment. Math. Univ. Carol.}}
\def\TAMS{{\it Trans.\ Amer.\ Math.\ Soc.}}
\def\JAMS{{\it J.~Amer.\ Math.\ Soc.}}
\def\PAMS{{\it Proc. Amer. Math. Soc.}}
\def\DM{{\it Discrete Math.}}
\def\CM{{\it Contemporary Math.}}
\def\GC{{\it Graphs and Combin.}}
\def\COM{{\it Combinatorica}}
\def\JGT{{\it J.~Graph Theory}}
\def\JAlgorithms{{\it J.~Algorithms}}
\def\SIAMDM{{\it SIAM J.~Disc.\ Math.}}
\def\CPC{{\it Combinatorics, Probability and Computing}}
\def\EJC{Electron.\ J.~Combin.}

\bibitem{4CT1} K. Appel and W. Haken, Every planar map is four colorable, Part I: discharging, {\it Illinois J. of Math.} {\bf21} (1977), 429--490.
\bibitem{4CT2} K. Appel, W. Haken, J. Koch, Every planar map is four colorable, Part II: reducibility, {\it Illinois J. of Math.} {\bf21} (1977), 491--567.
\bibitem{Bohme} T. B\"ohme, B. Mohar and M. Stiebitz, Dirac's map-color theorem for choosability, \JGT\ {\bf32} (1999), 327--339.
\bibitem{Erdos} P. Erd\H{o}s, A. Rubin, H. Taylor, Choosability in graphs, Proc. West Coast Conference on Combinatorics, Graph Theory and Computing, Arcata, {\it Congressus Numerantium} {\bf26} (1979), 125--157.
\bibitem{HutchOuterplanar} J. Hutchinson, On list-coloring extendable outerplanar graphs, {\it Ars Mathematica Contemporanea \bf5} (2012), 171--184.
\bibitem{PosPhD} L.~Postle,
$5$-list-coloring graphs on surfaces,
Ph.D. Dissertation, Georgia Institute of Technology, 2012.
\bibitem{PosThoHyperb} L.~Postle and R.~Thomas,
Hyperbolic families and coloring graphs on surfaces,
manusript.
\bibitem{ThomPlanar} C. Thomassen, Every planar graph is $5$-choosable, \JCTB\ {\bf62} (1994), 180--181.
\bibitem{ThomWheels} C. Thomassen, Exponentially many 5-list-colorings of planar graphs, \JCTB\ {\bf97} (2007), 571--583.
\bibitem{Viz} V.~G.~Vizing,
Coloring the vertices of a graph in prescribed colors (in Russian),
{\sl Diskret. Analiz} {\bf 29} (1976), 3--10.
\bibitem{Voigt} M. Voigt, List colourings of planar graphs, \DM\ {\bf120} (1993), 215--219.

\end{thebibliography}

\section*{Acknowledgment}
The result of this paper forms part of the doctoral dissertation~\cite{PosPhD} of the first author,
written under the guidance of the second author.

\baselineskip 11pt
\vfill
\noindent
This material is based upon work supported by the National Science Foundation.
Any opinions, findings, and conclusions or
recommendations expressed in this material are those of the authors and do
not necessarily reflect the views of the National Science Foundation.
\eject

\end{document}